\documentclass{amsart}

\usepackage{amsfonts, amsmath, amsthm, amssymb}
\usepackage[all]{xy}

\newtheorem{theorem}{Theorem}[section]
\newtheorem{lemma}[theorem]{Lemma}
\newtheorem{cor}[theorem]{Corollary}
\newtheorem{prop}[theorem]{Proposition}

\newtheorem{definitiontemp}[theorem]{Definition}
\newenvironment{definition}{\begin{definitiontemp}
\normalfont}{\end{definitiontemp}}

\theoremstyle{remark}
\newtheorem*{remark}{Remark}

\newcommand{\cp}{\mathbb{C}_p}
\newcommand{\qp}{\mathbb{Q}_p}
\newcommand{\lb}{\left[}
\newcommand{\rb}{\right]}
\newcommand{\D}{\mathcal{D}}
\newcommand{\Dx}{\left(\D\frac{1}{x}\right)\!}
\newcommand{\zp}{\mathbb{Z}_p}
\newcommand{\zpx}{\mathbb{Z}_p^{\times}}
\newcommand{\gp}{\Gamma_{\!\!p}}
\newcommand{\lgp}{\text{Log}\,\gp}
\newcommand{\f}{\text{exp}\,(x^p/p+x)}

\newcommand{\dind}[2]
{\genfrac{}{}{0pt}{}{#1}{#2}}

\title{Frobenius map and the $p$-adic Gamma function}
\author{Ilya Shapiro}

\begin{document}
\maketitle

\begin{abstract}
In this note we study the relationship between the power series expansion of the Dwork exponential and the Mahler expansion of the $p$-adic Gamma function.  We exploit this relationship to prove that certain quantities that appeared in our computations in \cite{quintic} can be expressed in terms of the derivatives of the $p$-adic Gamma function at $0$.  We then use this to prove the claim made in \cite{quintic} about the non-trivial off-diagonal entry in the Frobenius matrix of the mirror quintic threefold.
\end{abstract}

\section{Introduction.}

The paper \cite{quintic} was motivated by the recent use of $p$-adic methods in proving integrality results in the theory of topological strings (with \cite{ksv} a good starting point, see also \cite{sv1}, \cite{sv2} and \cite{vologodsky}).  Their main idea was the expression of certain physical quantities in terms of the Frobenius map on $p$-adic cohomology.

In \cite{quintic} we studied the Frobenius map on the cohomology of the mirror quintic family at the point of maximally unipotent monodromy.  More precisely, we looked at the limit of the Frobenius action at $\lambda=0$ of an appropriate quotient of the family of projective hypersurfaces $$\lambda(x_0^5+x_1^5+x_2^5+x_3^5+x_4^5)+x_0 x_1 x_2 x_3 x_4=0.$$ We used a Dwork-like complex, in the modern language it can be understood as the rigid cohomology with coefficients, to compute the Frobenius map directly.  This allowed us to bypass (in the quintic case) the machinery of motives that was needed in \cite{ksv}.

While successful in proving the vanishing of a certain Frobenius matrix element (this vanishing was the key component of \cite{ksv}), we were unable to demonstrate the conjectural relation of the only non-zero off-diagonal element to the $p$-adic zeta value at $3$.  This relation was communicated to us by the authors of \cite{ksv}, and numerically verified in \cite{quintic}.  The purpose of this note is to finally prove this conjecture for the case of the mirror quintic.  Namely, the non-trivial off-diagonal element is $$p^3\frac{2^3}{5^2}L_p(3,\omega^{-2})=(p^3-1)\frac{2^3}{5^2}\zeta_p(3)$$ where $L_p$ is the Kubota-Leopoldt $p$-adic $L$-function and $\omega$ is the Teichm\"{u}ller character.

The explanation for the presence of the zeta values in the Frobenius matrix is to be found among the conjectures in the theory of motives. Namely, the appearance of the $\zeta_p(3)$ factor (up to a rational number)
would follow if one knew that the limiting motive of the family is mixed Hodge-Tate.  However in the present concrete case of the mirror quintic family it can be understood in terms of the relationship between the Dwork exponential that appears in the definition of the Frobenius action, and the $p$-adic Gamma function that is closely related to the zeta values.

\subsection{Preliminaries.}

The $p$-adic Gamma function $\gp(x)$ is the $p$-adic analogue of the usual Gamma function.  It extends in a natural way, to the $p$-adic integers $\zp$, a slight modification of the factorial function on the integers.  More precisely, we have that $$\gp(x+1)=\begin{cases}
-x\gp(x) & |x|=1\\
-\gp(x) & |x|<1
\end{cases}$$ where $|\cdot|$ denotes the $p$-adic norm on $\zp$, i.e., $|x|<1$ if and only if $x\in p\zp$ and $|x|=1$ if and only if $x\in\zpx$.

We let $$(x)_n=x(x-1)...(x-n+1)$$ denote the \emph{falling} factorials.  We reserve the notation $(x)^n$ for the rising factorials, i.e., $(x)^n=x(x+1)...(x+n-1)$.  Note that $(x)^n=(-1)^n(-x)_n$.

By default $$g=\sum a_i x^i$$ is a power series with coefficients in $\cp$, the completion of the algebraic closure of $\qp$, the field of $p$-adic numbers.  We assume that $|a_i i!|\rightarrow 0$ which ensures that $g(x)$ is defined on $p\zp$, though perhaps not on $\zp$ itself. We call an expression $$\varphi=\sum a_i (x)_i$$ a Mahler series.  We again assume that $a_i\in\cp$ and $|a_i i!|\rightarrow 0$, so that $\varphi(x)$ is a continuous function on $\zp$.

For $\phi$ a function on $\zp$, let $$\nabla\phi(x)=\phi(x+1)-\phi(x)$$ be the forward difference operator.

As in \cite{quintic}, let $\D$ denote the formal differential operator of infinite order defined by $$\D=\sum_{i=0}^\infty \partial_x^i.$$

For a formal Laurent series $g=\sum a_i x^i$, we denote by $[g]_0$ the coefficient of $x^0$, i.e., $$[g]_0=a_0.$$

All sums are from $0$ to $\infty$ unless stated otherwise.

We call the power series $$f=\f=\sum B_i x^i$$ the Dwork exponential.  It defines a function on $p\zp$ though not on $\zp$.  Note that in our context $B_i$ are simply the coefficients in the expansion of $f$ and do not stand for the Bernoulli numbers.

In everything that follows, $p$ is assumed to be an \emph{odd} prime.

\section{From power series to Mahler series.}
We begin the study of the formal properties of the transform that converts power series to Mahler series.

\begin{definition}
Let $g=\sum a_i x^i$ be a power series, then denote by $$[g]=\sum a_i (-1)^i (x)_i$$ the associated Mahler series.
\end{definition}

Some properties of the above transform are immediate from the definition.

\begin{lemma}
Let $g$ be a power series. Then \begin{enumerate}\item $[x^n g]=(-1)^n(x)_n[g](x-n)$, \item $[\partial g]=-\nabla[g]$.\end{enumerate}
\end{lemma}

\begin{proof}
Let $g=\sum a_i x^i$, then \begin{align*}[x^n g]&=[\sum a_i x^{n+i}]=\sum a_i (-1)^{n+i}(x)_{n+i}\\&=\sum a_i (-1)^{n+i}(x)_n (x-n)_i\\&=(-1)^n(x)_n\sum a_i (-1)^i(x-n)_i=(-1)^n(x)_n[g](x-n).\end{align*}  Similarly, since $[\partial x^i]=\nabla(x)_i$, we obtain the second part.
\end{proof}

\begin{lemma}\label{inverse}
Let $g=\sum a_i x^i$ be a power series such that $|a_i i!|\rightarrow 0$, then $[g]$ defines a continuous function on $\zp$ and  $$g=e^x\sum[g](n)(-1)^n x^n/n!$$ as power series.
\end{lemma}

\begin{proof}
This is just a restatement of the results that can be found in \cite{robert}.
\end{proof}

The following lemma is the reason for the appearance of the Mahler series in the computation of the Frobenius matrix.  It relates the objects that arose in the cohomology computations, to the Stirling numbers of the second kind.  This connection was the link missing in \cite{quintic}.

\begin{lemma}
The following holds: $$\lb \Dx^s x^k\rb_0=\frac{(-1)^{s+k}}{s!}(x)_k^{(s)}(0).$$
\end{lemma}

\begin{proof}
We proceed by induction on $k$.  When $k=0$ the identity holds trivially. Consider $\lb \Dx^s x^{k+1}\rb_0$.  By \cite{quintic} we have that \begin{align*}\lb \Dx^s x^{k+1}\rb_0&=k\lb \Dx^s x^{k}\rb_0+\lb \Dx^{s-1} x^{k}\rb_0\\&=k\frac{(-1)^{s+k}}{s!}(x)_k^{(s)}(0)+\frac{(-1)^{s+k-1}}{(s-1)!}(x)_k^{(s-1)}(0).  \end{align*} But we also have that \begin{align*}\frac{(-1)^{s+k+1}}{s!}(x)_{k+1}^{(s)}(0)&=\frac{(-1)^{s+k+1}}{s!}((x)_{k}(x-k))^{(s)}(0)\\
&=\frac{(-1)^{s+k+1}}{(s-1)!}(x)_{k}^{(s-1)}(0)-k\frac{(-1)^{s+k+1}}{s!}(x)_{k}^{(s)}(0). \end{align*} Which completes the proof.
\end{proof}

We are now ready for the important corollary below.

\begin{cor}\label{Dtod}
Let $g$ be a power series.  Then $$\lb\Dx^s g\rb_0=\frac{(-1)^s}{s!}[g]^{(s)}(0).$$
\end{cor}
\begin{proof}
Suppose that $g=\sum a_i x^i$ so that $[g]=\sum a_i (-1)^i (x)_i$.  Thus \begin{align*}\lb\Dx^s g\rb_0&=\sum a_i\lb\Dx^s x^i\rb_0\\
&=\sum a_i\frac{(-1)^{i+s}}{s!}(x)_i^{(s)}(0)=\frac{(-1)^s}{s!}[g]^{(s)}(0).\end{align*}
\end{proof}

\subsection{The Dwork exponential and the Gamma function.}
Here we give an interpretation of the transform of the Dwork exponential in terms of the $p$-adic Gamma function. Almost certainly, the formula below is not new, however we were unable to find it in the literature. We then proceed to apply the above results in the next section to shed light on the mysterious objects $$\lb\Dx^s f\rb_0$$ that appeared in \cite{quintic}.

We start with the computation of $[f]$.  There are some formulas that one can find in the literature that hint at the relationship that we establish below.  For example in \cite{rodvill} one can find the following (see also \cite{rob2}): $$\gp(-a+px)=\sum_{k\geq 0} p^k B_{a+kp}(x)^k$$ for $0\leq a<p$ and $x\in\zp$.  And even more enticingly, we learn from \cite{robert} that for $x\in\zp$ one has $$\gp(x+1)=\lb\frac{1-x^p}{x-1}f\rb.$$  In fact we can (we mention two proofs) use the latter in the proposition below.

\begin{remark}
Even though $f$ itself does not define a function on all of $\zp$, one sees that $[f]$ is at least a continuous function on $\zp$.
\end{remark}

\begin{prop}\label{dworktogamma}
Let $f=\f$, then $$[f](x)=\begin{cases}\Gamma_p(x) & |x|<1\\0&|x|=1\quad. \end{cases}$$
\end{prop}

\begin{proof}
Let $g$ be a power series such that $[g]=\gp(x+1)$.  Then by the formula in \cite{robert} mentioned above we see that $(1-x^p)f=(x-1)g$.  Thus $[f]-[x^p f]=[xg]-[g]$ and so $[f]+x[x^{p-1}f](x-1)=-x\gp(x)-\gp(x+1)$.  Observe that $x^{p-1}f=\partial f-f$ so that \begin{align*}[x^{p-1}f](x-1)&=[\partial f](x-1)-[f](x-1)\\&=-\nabla[f](x-1)-[f](x-1)=-[f]. \end{align*} We conclude that $[f]-x[f]=-x\gp(x)-\gp(x+1)$.  The latter is $0$ if $|x|=1$ and $(1-x)\gp(x)$ if $|x|<1$. From the definition we see that $[f](1)=B_0-B_1=0$.  The result follows.
\end{proof}

One can also proceed more directly (as suggested by A. Mellit).

\begin{proof}
By Lemma \ref{inverse} we have that $f(x)=e^x\sum[f](n)(-1)^n x^n/n!$, i.e., $$e^{x^p/p}=\sum[f](n)(-1)^n x^n/n!.$$  By comparing coefficients we see that if $p\nmid n$ then $[f](n)=0$ and if $n=ap$ then $[f](n)=\frac{(-1)^a(ap)!}{p^a a!}=\gp(n)$.  We are done by continuity.
\end{proof}

\section{Consequences for $f$.}

After establishing that the transform of the Dwork exponential is essentially the $p$-adic Gamma function we may now derive a multitude of formulas involving the coefficients $B_k$ using the methods of the previous section.  We start with the most important application.

\begin{prop}\label{mystsolved}
Let $f=\f$. Then $$\lb\Dx^s f\rb_0=\frac{(-1)^s}{s!}\gp^{(s)}(0).$$
\end{prop}

\begin{proof}
Combine Proposition \ref{dworktogamma} with Corollary \ref{Dtod}.
\end{proof}

The corollary below generalizes the proposition.  The objects it describes also appear in the Frobenius calculation.  We were mostly interested in the case of $n=p$; the other cases can in principle be avoided if one uses certain general properties that the Frobenius map is known to satisfy.  Thus the corollary's interest to us in its full generality is mainly esthetic.

\begin{cor}\label{fun}
Let $f=\f$, we have $$\lb\Dx^s x^n f\rb_0=\begin{cases} 0 & p\nmid n\\\dfrac{(-1)^s}{s!}\left(p^a(x/p)_a\gp(x)\right)^{(s)}(0)
 & n=ap \quad.\end{cases}$$
\end{cor}

\begin{proof}
As above we combine Proposition \ref{dworktogamma} with Corollary \ref{Dtod}.  Namely, $$\lb\Dx^s x^n f\rb_0=\frac{(-1)^s}{s!}[x^n f]^{(s)}(0)=\frac{(-1)^s}{s!}((-1)^n(x)_n[f](x-n))^{(s)}(0).$$  If $p\nmid n$ then near $x=0$, the function $[f](x-n)$ is identically $0$; this proves the first part.  Now let $n=ap$, then $$\lb\Dx^s x^{ap} f\rb_0=\dfrac{(-1)^s}{s!}((-1)^a(x)_{a\!p}\,\gp(x-ap))^{(s)}(0).$$  However one checks that for $|x|<1$, we have $$(x)_{a\!p}\,\gp(x-ap)=(-1)^a p^a(x/p)_a\gp(x)$$ and the result follows.
\end{proof}

\begin{remark}
So that for $n=p$ we have: $$\lb\Dx^s x^p f\rb_0=\dfrac{(-1)^s}{s!}(\gp(x)x)^{(s)}(0)=\dfrac{(-1)^s}{(s-1)!}\gp^{(s-1)}(0).$$
\end{remark}

The following result about the coefficients $B_k$ of the Dwork exponential partially intersects with the one found in \cite{katz}, in particular the latter does not address the case of $p\nmid n$. In \cite{katz}   considerations very different from ours are used.

\begin{cor}\label{funapplied}
Recall that $f=\f=\sum B_k x^k$, then for $n>0$
$$\sum B_k(n+k-1)!=\begin{cases}0 & p\nmid n\\(-1)^a p^{a-1} (a-1)! & n=ap\quad .\end{cases}$$
\end{cor}
\begin{proof}
Note that $$\sum B_k(n+k-1)!=\lb\Dx x^n f\rb_0.$$  By Corollary \ref{fun} this proves the $p\nmid n$ case.  If $n=ap$ then \begin{align*}\lb\Dx x^{a\!p} f\rb_0&=-(\gp(x)x(x-p)...(x-(a-1)p))'(0)\\
&=-(\gp(x)\{p^{a-1}(-1)^{a-1}(a-1)!x+\text{higher powers of}\,\,x\})'(0)\\&=(-1)^a p^{a-1} (a-1)!\,\gp(0). \end{align*}
\end{proof}

In fact we can handle the case of $n=0$ as well, though the answer is less explicit.

\begin{cor}
With $B_k$ as above, $$\sum_{k\geq 1} B_k (k-1)!=-\gp'(0).$$
\end{cor}
\begin{proof}
Apply Proposition \ref{mystsolved} to the observation $\sum B_k (k-1)!=\lb\Dx f\rb_0$.
\end{proof}

Proposition \ref{dworktogamma} evaluates $\sum B_k(-1)^k(x)_k$.  In particular, when $|x|<1$ the sum is equal to $\gp(x)$.  A natural question to ask is what happens when we shift the coefficients $B_k$.  More precisely, what can one say about $\sum B_{k-s}(-1)^k(x)_k$ for $s\geq 0$?

\begin{cor}
Let $s\geq 0$ and $f=\sum B_k x^k$, then $$\sum B_{k-s}(-1)^k(x)_k=\begin{cases}0 & |x-s|=1\\(-1)^s(x)_s\,\gp(x-s) & |x-s|<1\quad.\end{cases}$$
\end{cor}

\begin{proof}
Observe that $$\sum B_{k-s}(-1)^k(x)_k=(-1)^s(x)_s\sum B_{k-s}(-1)^{k-s}(x-s)_{k-s}=(-1)^s(x)_s[f](x-s).$$ We are done by Proposition \ref{dworktogamma}.
\end{proof}

\begin{remark}
Since $\sum B_k (n+k-1)!=\sum B_{k-(n-1)}k!=\sum B_{k-(n-1)}(-1)^k(-1)_k$ the above corollary provides an alternative proof of Corollary \ref{funapplied}.
\end{remark}

Note that $B_k$s can be used to compute the values of the $p$-adic Gamma function at the invertible $p$-adic integers as well.  Namely, if $|x|<1$ and $0<a<p$ then $$\gp(a+x)=(-1)^a(x+1)_{a-1}\sum B_k (-1)^k(x)_k.$$

\section{Connection to the $p$-adic $L$-functions.}
In this section we conclude the proof of the claim made in \cite{quintic} by relating the derivatives of the $p$-adic Gamma function at $0$ to certain values of the  Kubota-Leopoldt $L$-function.  The latter can be expressed in terms of the $p$-adic zeta values by \cite{furusho}.

\begin{lemma}\label{lfandlg}
Let $s\geq 2$, then $$L_p(s,\omega^{1-s})=\frac{(-1)^s}{(s-1)!}(\lgp)^{(s)}(0).$$
\end{lemma}

\begin{proof}
By definition (see \cite{diamond} for example) $$L_p(s,\omega^{1-s})=\frac{1}{s-1}\lim_{k\rightarrow \infty}\frac{1}{p^{k+1}}\sum_{\dind{n=1}{p\nmid n}}^{p^{k+1}}\frac{1}{n^{s-1}}.$$   The latter is simply $\dfrac{1}{s-1}\displaystyle\int_{\zpx}t^{1-s}\, dt$.  Note that the integral vanishes for even $s\geq 2$, and $\lgp$ is odd (see \cite{robert} for both claims).  Thus the case of even $s$ is proven.  Now let $s=2m+1$ with $m\geq 1$.  Then $$L_p(2m+1,\omega^{-2m})=\frac{1}{2m}\int_{\zpx}t^{-2m}\, dt$$ while (see \cite{robert}) $$(\lgp)^{(2m+1)}(0)=-(2m-1)!\int_{\zpx}t^{-2m}\, dt$$ and we are done.
\end{proof}

The above provides an extension (via a minor modification) of the formula found in \cite{koblitz} to the case of the trivial character.  Namely, for $s\geq 2$ and $\chi$ a primitive Dirichlet  character of conductor $d$ we have $$L_p(s,\chi\omega^{1-s})=\frac{(-d)^{-s}}{(s-1)!}\sum_{0\leq a<d}\chi(a)(\lgp)^{(s)}\left(\frac{a}{d}\right).$$  The formula in \cite{koblitz} is valid for $s\geq 1$ but only for $\chi\neq\chi_{\text{triv}}$.  The modification is the extension of the range of summation above to include $a=0$.

\subsection{Evaluation of $\Delta_3$.}
In \cite{quintic} we defined $$\Delta_s=\lb\Dx^s f\rb_0-\frac{\lb\Dx f\rb_0^s}{s!},$$ and noted that $$\Delta_2=0$$ followed from the general properties of the Frobenius map.  Now that we have an interpretation in terms of the Gamma derivatives, this fact follows from the observation (see \cite{robert} for example) that $\lgp$ is odd.  The following theorem addresses the question of $\Delta_3$ whose value was stated only conjecturally in \cite{quintic}.

\begin{theorem}
We have that $$\Delta_3=\frac{L_p(3,\omega^{-2})}{3}\,.$$
\end{theorem}

\begin{proof}
Using the fact that $\lgp$ is odd we see that $$(\lgp)^{(3)}(0)=\gp'''(0)-(\gp'(0))^3.$$ By definition $\Delta_3=\lb\Dx^3 f\rb_0-\lb\Dx f\rb_0^3/3!$, and by Proposition \ref{mystsolved} the latter is $-\frac{1}{6}(\gp'''(0)-(\gp'(0))^3)=-\frac{1}{6}(\lgp)^{(3)}(0)$.  We are done by Lemma \ref{lfandlg}.
\end{proof}

\begin{remark}
The connection between the zeta values and $L$-functions can be found in \cite{furusho}. More precisely, we can take advantage of the formula $\zeta_p(s)=\dfrac{p^s}{p^s-1}L_p(s,\omega^{1-s})$ to obtain $$\Delta_3=\left(1-\frac{1}{p^3}\right)\frac{\zeta_p(3)}{3}\,.$$
\end{remark}

We conclude that the non-trivial off-diagonal Frobenius matrix entry in \cite{quintic} is indeed $$p^3\frac{24}{25}\Delta_3=p^3\frac{2^3}{5^2}\,L_p(3,\omega^{-2})=(p^3-1)\frac{2^3}{5^2}\,\zeta_p(3)$$ as promised.

\bigskip
\noindent{\bf Acknowledgments.} We are indebted to P. Candelas who noticed the similarity between formulas appearing in \cite{quintic} and in his own work. His observation, based on this similarity, that one should be able to express in a very precise way the mysterious quantity in \cite{quintic} in terms of the derivatives of the $p$-adic Gamma function at $0$ was the starting point of this paper.  The appearance of zeta of $3$ was predicted by M. Kontsevich and V. Vologodsky, and we thank V. Vologodsky for his patience in providing a detailed explanation. We thank M. Hadian Jazi and A. Mellit  for useful discussions.  We would also like to thank A. Schwarz for pointing out some references, as well as for encouragement and advice.    We appreciate the hospitality of MPIM Bonn and
University of Waterloo where parts of this paper were written.

\medskip
\noindent Max-Planck-Institut f\"{u}r Mathematik, Bonn
\newline \emph{E-mail address}:
\textbf{ishapiro@mpim-bonn.mpg.de}

\end{document}